\documentclass[a4paper,11pt,reqno]{amsart}
\usepackage[utf8]{inputenc}
\usepackage{mathtools}
\usepackage{amssymb}
\usepackage{amsfonts}
\usepackage{amsthm}
\usepackage{bbm}
\usepackage{fullpage}
\usepackage{color}
\usepackage{hyperref}
\usepackage{stmaryrd}
\usepackage{enumitem}
\usepackage{accents}
\mathtoolsset{showonlyrefs}

\author{Hubert Lacoin}
\address{IMPA, Estrada Dona Castorina 110,
Rio de Janeiro RJ-22460-320- Brasil}
\email{lacoin@impa.br}

\newtheorem{theorem}{Theorem}[section]

\newtheorem{lemma}[theorem]{Lemma}
\newtheorem{proposition}[theorem]{Proposition}

\newtheorem{propx}{Proposition}

\newcommand{\bbE}{{\ensuremath{\mathbb E}} }

\newcommand{\bbL}{{\ensuremath{\mathbb L}} }

\newcommand{\bbN}{{\ensuremath{\mathbb N}} }

\newcommand{\bbP}{{\ensuremath{\mathbb P}} }

\newcommand{\bbR}{{\ensuremath{\mathbb R}} }

\newcommand{\bbZ}{{\ensuremath{\mathbb Z}} }

\renewcommand{\epsilon}{\varepsilon}

\newcommand{\gep}{\varepsilon}       

\newcommand{\go}{\omega}

\newcommand{\gO}{\Omega}
\newcommand{\gl}{\lambda}

\newcommand{\cF}{{\ensuremath{\mathcal F}} }

\newcommand{\cP}{{\ensuremath{\mathcal P}} }

\newcommand{\bP}{{\ensuremath{\mathbf P}} }

\newcommand{\bE}{{\ensuremath{\mathbf E}} }

\newcommand{\ind}{\mathbf{1}}

\newcommand{\lint}{\llbracket}
\newcommand{\rint}{\rrbracket}

\begin{document}



\title{A short proof of diffusivity for the directed polymers in the weak disorder phase}

\begin{abstract}
 We provide a new short and elementary proof of diffusivity for directed polymers in a random environment in the weak disorder phase.
\end{abstract}

 \maketitle

\section{Model and results}

\subsection{The model}
We consider $(X_i)_{i\ge 0}$ { the simple} random walk on $\bbZ^d$ starting from $0$ ({ with transition probability given by $P(X_{n+1}=x \ | \ X_n=y)=\frac{1}{2d}\ind_{\{|x-y|_1=1\}}$ where $|\cdot|_1$ is the $\ell_1$ norm on $\bbR^d$}) and $\zeta:=(\zeta_{k,x})_{(k,x)\in  \bbN \times \bbZ^d}$ be a collection of nonnegative i.i.d.\ random variables of mean $1$ (and probability distribution denoted by $\bbP$).
 We let  $\gO$ be the set of random walk trajectories (equiped with product $\sigma$-algebra) and $L^{\infty}(\gO)$  the set of bounded measurable functions on $\gO$. For a fixed realization of $\zeta$, and $n\ge 0$, we define a random measure $W^{\zeta}_n$ on $\gO$ by setting for $f\in L^{\infty}(\gO)$ 
\begin{equation}
 W^{\zeta}_n(f) = E\bigg[ \Big(\prod_{i=1}^n \zeta_{i,X_i}\Big)f(X)\bigg].
\end{equation}
We also write for $A\subset \gO$ measurable,  $W^{\zeta}_n(A)= W^{\zeta}_n(\ind_{A})$.
With some abuse of notation we let  $W^{\zeta}_n:= W^{\zeta}_n(\gO)$ denote the total mass of the measure ({ note in particular that $\bbE[W_n]=1$}).
On the event $W^{\zeta}_n>0$, we introduce the probability $P^{\zeta}_n$ on $\gO$ by setting
\begin{equation}
P^{\zeta}_n(A)=   W^{\zeta}_n(A) / W^{\zeta}_n.
\end{equation}
The probability $P^{\zeta}_n$ is known as the \textit{directed polymer in the  random environment $\zeta$}.
One can verify (see \cite[Lemma 1]{B89}) that the sequence $(W^{\zeta}_n)_{n\ge 0}$ is a martingale w.r.t.\ the filtration $(\cF_n)_{n\ge 0}$ defined by
$\cF_{n}:= \sigma( (\zeta_{k,x})_{(k,x)\in \lint 1,n\rint \times \bbZ^d})$ (where $\lint a,b \rint=[a,b]\cap \bbZ$).
Hence $W^{\zeta}_n$  converges almost surely as $n$ tends to infinity.
We set
 $W^{\zeta}_\infty:= \lim_{n\to \infty} W^{\zeta}_n$. We are interested in the regime  known as \textit{weak disorder} where the limit $W^{\zeta}_\infty$ is nontrivial, that is, not identically equal to zero. We recall here a standard result (see \cite[Proposition 3.1]{CY06} and \cite[Proposition 2.3]{HL12}) concerning this regime. A short proof is included  for completeness in { Section \ref{aba}}.

\begin{propx}\label{clacos}
 We have the following equivalence
 \begin{itemize}
  \item [(i)] $(W^{\zeta}_n)_{n\ge 1}$ is uniformly integrable and converges in $\bbL^1$.
  \item [(ii)] $\bbP(  W^{\zeta}_\infty>0 )>0$.
 \end{itemize}
If { $(i)$ and $(ii)$ hold}, we further have $\bbP(W^{\zeta}_{\infty}>0)= \bbP( \forall n>0,\ W^\zeta_n>0)$,
 in particular $\bbP(W^{\zeta}_{\infty}>0)=1$, if $\bbP(\zeta_{k,x}>0)=1$.

\end{propx}

\subsection{An invariance principle for $P^{\zeta}_n$}

We extend the trajectory of $X$ to $\bbR_+$ by linear interpolation setting, for $s\in \bbR_+$, $X_s:=(s-\lfloor s \rfloor)  X_{\lfloor  s \rfloor}+ (\lfloor s \rfloor+1-s)  X_{\lfloor s \rfloor +1} $ ($\lfloor s\rfloor$ denotes the integer part of $s$).
The diffusively rescaled trajectory $(X^{(n)}_t)_{t\in[0,1]}$ is defined by
$ X^{(n)}_t= \sqrt{\frac{1}{n}} X_{nt}.$

\medskip

Note that $X^{(n)}$ is an element of
 $C_0([0,1])$ - the set of continuous function $[0,1] \to \bbR^d$ taking value  $0$ at $0$. We equip $C_0([0,1])$ with  the topology of uniform convergence.
We let $\mathcal C$ denote the set of bounded continuous  functions $C_0([0,1])\to \bbR$.
 Under $P$, Donsker's Theorem states that the distribution of  $(X^{(n)}_t)_{t\in[0,1]}$ converges weakly to that of a standard Brownian motion.
It has been established in \cite{CY06} that this scaling behavior also holds under $P^{\zeta}_n$.
The objective of this note is to provide a new short proof of this result.

 \begin{theorem}\label{main}

When $\bbP(  W^{\zeta}_\infty>0 )>0$, then for any  $\varphi\in \mathcal C$, we have the following convergence in probability under $\bbP':=\bbP(\cdot  \ | \ W^{\zeta}_\infty>0)$
\begin{equation}\label{convwiener}
 \lim_{n\to \infty} E^{\zeta}_n\left[\varphi(X^{(n)})\right]=   \bE\left[ \varphi(B)\right].
\end{equation}
 where $\bE$ is the expectation w.r.t.\ a $d$-dimensional Brownian Motion $(B_t)_{t\in[0,1]}$  with covariance $\frac{1}{d}I_d$ where $I_d$ is the identity matrix.
  \end{theorem}

\subsection{Almost-sure subsequencial convergence}

Since $C_0([0,1])$ is separable, the topology associated with weak convergence on $\cP(C_0([0,1]))$ -- the set of probability measures on $C_0([0,1])$ -- is metrizable.
For instance (cf.\ \cite[Appendix III, Theorem 5]{Bill}) it is generated by the Lévy-Prokhorov metric which is defined by
\begin{equation}
d_{\mathrm{LP}}(\pi,\pi'):= \inf\{ \gep\in(0,1] \ : \ \forall A\subset C_0([0,1]), \
\pi(A)\le \pi'(A_{\gep})+\gep \text{ and } \pi'(A)\le \pi(A_{\gep})+\gep\}
\end{equation}
where $A$ above is a Borel subset of  $C_0([0,1])$ and $A_{\gep}$ is the $\gep$-neighborhood of $A$, $$A_{\gep}:= \{ \theta \in C_0([0,1]): \exists \theta'\in A, |\theta-\theta'|_{\infty}<\gep\}.$$
{The convergence \eqref{convwiener} applied to a finite family of functions $(\varphi_i)_{i\in I}$ implies that\\ $\lim_{n\to \infty}\bbP'\left(P^{\zeta}_n\left(X^{(n)} \in \cdot\right)\in U\right)=1$, for a all $U$ in  neighborhood  basis of  $\bP(B\in \cdot)$ for the weak topology.  In particular}
\begin{equation}\label{inproba}
 \forall \gep>0, \quad \lim_{n\to \infty} { \bbP'}\left( d_{\mathrm{LP}}\left(P^{\zeta}_{n}(X^{(n)}\in \cdot), \bP(B\in \cdot)\right) <\gep \right)=1.
\end{equation}
This implies that one can find a sequence $(n_k)_{k\ge 0}$ tending to infinity such that the following weak convergence of probability measures holds $\bbP'$ almost surely
\begin{equation}
 \lim_{k\to \infty} P^{\zeta}_{n_k}(X^{(n_k)}\in \cdot)= \bP(B\in \cdot).
\end{equation}

\subsection{A by-product of our proof}
\noindent We end this result section with a  comparision result which states that  $P^{\zeta}_n(B)$ and $P(B)$  are close in probability if $B$ does not depend in the first few steps of the walk $X$.
More precisely, if we let $\mathcal G_m:= \sigma((X_{m+i}-X_m)_{i\ge 0})$ the following holds.
 \begin{proposition}\label{auxil}
When $\bbP(  W^{\zeta}_\infty>0 )>0$, we have
\begin{equation}\label{tyui}
  \lim_{m\to \infty} \sup_{B\in \mathcal G_m} \sup_{n\ge 0} \bbE'\left[ |P^{\zeta}_n(B)-P(B)|\right]=0.
\end{equation}
 \end{proposition}
 \noindent The above does not imply that the restriction of $P^{\zeta}_n$ and $P$ to $\mathcal G_m$ are close  in total-variation and in fact as a consequence  \cite[Theorem 6.2]{CY06}
$\bbP'$-a.s.\ we have for any $m\ge 0$,
 $\lim_{n\to \infty}\sup_{B\in \mathcal G_m} |P^{\zeta}_n(B)-P(B)|=1$.
 However \eqref{tyui} implies directly that
 \begin{equation}\label{finalimplication}
   \lim_{m\to \infty} \sup_{B\in \mathcal G_m} \sup_{n\ge 0} |\bbE'[P^{\zeta}_n(B)]-P(B)|=0.
 \end{equation}
In words, as $m\to \infty$ the distribution of $(X_{m+k} - X_m)_{k\ge 0}$ under the averaged probability  $\bbE' P^{\zeta}_n$ converges in total-variation   to that of the simple random walk \textit{uniformly in $n$}.

\subsection{Comments on main the result and its proof}

The main result in this paper has been first proved as  \cite[Theorem 1.2]{CY06} , in the setup where $\zeta_{k,x}=e^{\beta \go_{k,x}-\gl(\beta)}$, $\go_{k,x}$ is a field of i.i.d.\ real valued random variables indexed by $\bbN\times \bbZ^d$, $\beta\in \bbR$ and $\gl(\beta):=\log \bbE[e^{\beta \go_{k,x}}]$ (we refer to the introduction of \cite{CY06} for the history of this result). The assumption is made that
$\gl(\beta)$ is finite for all $\beta$.
This corresponds to assuming that $\zeta_{x,k}$ has moments of all orders (positive and negative).

\medskip

The result was then extended in \cite{NakaLSE} to a wider setup   where  the coefficients $\zeta_{x,k}$ are replaced by random  matrices with nonnegative coefficients. Besides, the setup of \cite{NakaLSE} allows for spatial correlation and reduces the integrability assumption, assuming only that $\bbE\left[ (\zeta_{x,k})^2\right]=1$. { In \cite{Wei16}, a variant of the result for the case when $X$ is a random walk on $\bbZ$ in the domain of attraction of an $\alpha$-stable process is presented (in that case $\varphi$ in \eqref{convwiener} is a continuous function on the space of càdlàg path $D[0,1]$ and $B$ is replaced by an $\alpha$-stable process).}

\medskip

The arguments presented in \cite{CY06,NakaLSE} rely strongly on the existence of a second moment for $\zeta_{x,k}$. Since there exists environments $\zeta$ with infinite second moment such that $\bbP(W^{\zeta}_{\infty}>0)>0$ (see for instance \cite[Theorem 1.3]{V21}), our proof provides a genuine extension of the domain of validity of the result. On the other hand our proof does not fully  extend to the setup considered in \cite{NakaLSE} since it requires positive correlation to make use of Harris/FKG inequality.

\medskip

For simplicity, we restricted this note to the case where the reference walk is the simple random walk on $\bbZ^d$. Only minor modifications would be required to treat the case of a general walk $(X_k)$ with increments { either} in the domain of attraction of Brownian Motion { or of an $\alpha$-stable process}.  In the same manner, the proof can also be adapted to continuous or semicontinuous models constructed from a Poisson process  or a Gaussian white noise. This includes the continuous time directed polymer on $\bbZ^d$ with Brownian environment  \cite{CH06}, the Brownian Polymer in a Poisson environment  \cite{Cpoisson} and the Brownian Polymer in an environment given by a space-time Gaussian white noise which is mollified in space \cite{MSZ18}.

\medskip

\section{Proofs}

\subsection{Main Lemma}
For better readability we drop the dependence in $\zeta$ in the notation. Our main technical input is the following result which controls the $\bbL^1$ rate of convergence of $(W_n(g))_{n\ge 0}$ uniformly in $g$.

\begin{lemma}\label{techos}
 Given a function $g: \gO\to {[0,1]}$,  for any $m\le n$ we have
\begin{equation}\label{plupetit}
\bbE\left[|W_n(g)-W_m(g)| \right] \le  \bbE\left[ |W_n-W_m| \right]
\end{equation}

\end{lemma}

 \begin{proof}

Using the notation $x_+= \max(x,0)$. we are going to prove that
\begin{equation}\label{plupetit2}
 \bbE\left[(W_n(g)-W_m(g))_+ \right] \le  \bbE\left[ (W_n-W_m)_+ \right].
\end{equation}
which is equivalent to \eqref{plupetit} since $|x|=2x_+-x$.
We set $\bar g= 1-g$, we have
\begin{equation}\begin{split}
\bbE\left[ (W_n-W_m)_+ \right]&\ge\bbE\left[ (W_n-W_m)_+ \ind_{\{W_n(g)\ge W_m(g)\}}\right]\\
&\ge \bbE\left[ (W_n(g)-W_m(g))_+\right] + \bbE\left[(W_n(\bar g)-W_m(\bar g)) \ind_{\{W_n(g)\ge W_m(g)\}}  \right].
\end{split}
 \end{equation}
In the second inequality we simply used the fact that $(a+b)_+\ge a + b$. To conclude we simply need to show that
\begin{equation}\label{final}
\bbE\left[(W_n(\bar g)-W_m(\bar g)) \ind_{\{W_n(g)\ge W_m(g)\}}  \right]\ge 0.
\end{equation}
Since $\bar g,g \ge 0$, when fixing the value of $(\zeta_{k,x})_{k\in \lint 1,m\rint,x\in \bbZ^d}$ , $W_n(\bar g)-W_m(\bar g)$ is an increasing function of the other coordinates $(\zeta_{k,x})_{k\ge m+1,x\in \bbZ^d}$ and so is  $\ind_{\{W_n(g)\ge W_m(g)\}}$. Using the Harris/FKG inequality which states that increasing functions on a product space are positively correlated (see for instance \cite[Theorem 3]{preston} for a general statement in the continous setup) 
we obtain that  \begin{multline}
 \bbE\left[ (W_n(\bar g)-W_m(\bar g)) \ind_{\{W_n(g)\ge W_m(g)\}} \ | \ \cF_m  \right]\\
 \ge \bbE\left[ (W_n(\bar g)-W_m(\bar g)) \ | \ \cF_m\right] \bbE\left[ \ind_{\{W_n(g)\ge W_m(g)\}} \ | \ \cF_m \right]=0.
\end{multline}
which concludes the proof of \eqref{final}.
 \end{proof}

\subsection{Proof of Theorem \ref{main}}
We assume that $\bbP( W_{\infty}>0)>0$.
By scaling we may assume that $\varphi$ { takes values in $[0,1]$}.
We set $m=m_n=\lfloor n^{1/4} \rfloor$. Using  Lemma \ref{techos} for $g=\varphi^{(n)}$ defined by $\varphi^{(n)}(X)=\varphi(X^{(n)})$  we have
\begin{equation}\label{appli}
\lim_{n\to \infty}\bbE\left[ \left|W_n(\varphi^{(n)})- W_m(\varphi^{(n)})\right|\right]\le \lim_{n\to \infty}\bbE\left[ |W_n-W_m| \right]=0,
\end{equation}
where the last equality follows from Proposition \ref{clacos}. This implies in turn that {
\begin{equation}\label{erty}
 \lim_{n\to \infty}\left|E^\zeta_n[\varphi(X^{(n)})]- E^{\zeta}_m[\varphi(X^{(n)})]\right|=
  \lim_{n\to \infty}\left| \frac{W_n(\varphi(X^{(n)}))}{W_n}- \frac{W_m(\varphi(X^{(n)}))}{W_m}\right|=0
\end{equation}
in probability under $\bbP'$  since the denominators both converge a.s.\ to $W_{\infty}$ which is positive under $\bbP'$ by definition}.
To conclude we simply observe that $\bbP'$-almost surely we have
\begin{equation}\label{firststep}
\lim_{n\to \infty}|E^{\zeta}_m[\varphi(X^{(n)}) ]-\bE[\varphi(B)]|=0,
  \end{equation}
 The above holds because {$P^{\zeta}_m(X^{(n)}\in \cdot)$} converges weakly to { $\bP\left( B\in \cdot \right)$}. Indeed, { since the environment only affects the distribution of the first $m$ increments under $P^{\zeta}_m$, we can couple $(X_k)$ with dsitribution $P^{\zeta}_m$ with a simple random walk $(Y_k)$ in such a way that the increments of $X$ and $Y$ coincide starting from step $m+1$. With this coupling, $|X_k-Y_k|_1\le m$ for every $k\ge 0$, and thus $\max_{t\in [0,1]}|X^{(n)}_t-Y^{(n)}_t|_1\le m/\sqrt{n}$. Since $m/\sqrt{n}=o(1)$ with our choice of $m$,}
 \eqref{firststep} follows from Donsker's Theorem for $Y^{(n)}$. \qed

 {
 \subsection{Proof of Proposition \ref{auxil}}

 From Lemma \ref{techos} and Proposition \ref{clacos} we have
\begin{equation}\label{troop}
 \lim_{m\to \infty} \sup_{n\ge m} \sup_{A\in \mathcal G_0} \bbE\left[ |W_n(A)-W_m(A)|\right]=0.
\end{equation}
From the definition of $P^{\zeta}_n$ we have
\begin{equation}
 |P^{\zeta}_n(A)-P^{\zeta}_m(A)|\le \ind_{\{ \min(W_n,W_m)>\gep\}}+ \gep^{-1}\left( |W_n-W_m| + |W_n(A)-W_m(A)|\right)
\end{equation}
and hence \eqref{troop} implies that
\begin{equation}
  \lim_{m\to \infty} \sup_{n\ge m} \sup_{A\in \mathcal G_0} \bbE'\left[ |P^{\zeta}_n(A)-P^{\zeta}_m(A)|\right]=0.
\end{equation}
We obtain the result by applying the above for $B\in \mathcal G_m$ (since we have $P^{\zeta}_n(B)=P(B)$ for $n\le m$ we can consider the sup for $n\ge 0$). \qed
}
\section{Proof of Proposition \ref{clacos}}\label{aba}
\noindent The proposition follows from the following
\begin{equation}\label{tropic}
  \bbE[W_\infty]\in\{0,1\} \quad \text{ and } \quad \bbP(W_{\infty}>0) \times \bbP\left(\{\forall n, W_{n}>0\} \setminus \{W_{\infty}>0 \}\right)=0.
  \end{equation}
The first claim implies the equivalence of $(i)$ and $(ii)$ via Scheffé's Lemma and the second one implies the other implication: since the sequence of events $(\{W_{n}>0\})_{n\ge 0}$ is noncreasing for the inclusion order, we have  $\{W_{\infty}>0\}\subset \{\forall n, W_{n}>0\} $.
Our starting point is the following decomposition of the infinite volume partition function
\begin{equation}\label{crops}
 W_{\infty}= \sum_{z\in \bbZ^d}  W_n(X_n=z) \theta_{n,z} W_\infty.
\end{equation}
where the shifted environment  $\theta_{n,z} \zeta$ is defined by $(\theta_{n,z} \zeta)_{k,x}= \zeta_{n+k,z+x}$ and 
$\theta_{n,z} W^{\zeta}_\infty:= W^{\theta_{n,z}\zeta}_\infty.$
Taking conditional expectation and using translation invariance for the environment we have
\begin{equation}
 \bbE\left[ W_{\infty} \ | \ \cF_n \right]=\sum_{x\in \bbZ^d}   W_n(X_n=z) \bbE\left[\theta_{n,z} W_\infty\right]=  W_n \bbE\left[ W_{\infty}\right].
\end{equation}
Taking the limit in $n\to \infty$ one finally obtains $W_{\infty}= W_{\infty }\bbE\left[ W_{\infty}\right]$ and hence $\bbE[W_\infty]= \bbE[W_\infty]^2$, proving the first point.
For the second result we deduce from \eqref{crops} that
\begin{equation}\label{trop}\begin{split}
  \bbE\left[ \ind_{\{W_\infty>0\}} \ | \ \cF_n \right]& = \bbP\left[ \exists z,\  \theta_{n,z} W_\infty>0 \ ; \   W_n(X_n=z)>0 \ | \ \cF_n\right]\\ &\ge   \ind_{\{W_n>0\}}   \bbP(W_\infty>0).
 \end{split}\end{equation}
To obtain the last inequality above, we equip $\bbZ^d$ with the lexicographical order and { on the event $\{W_n>0\}$} we set
 $Z_n:=\min\{  z\in\bbZ^d : W_n(X_n=z)>0\}$. When $W_n>0$ we have
 \begin{equation}
  \bbP\left[ \exists z,\  \theta_{n,z} W_\infty>0 \ ; \  W_n(X_n=z)>0 \ | \ \cF_n\right]\ge \bbP\left[ \theta_{n,Z_n} W_\infty>0   \ | \ \cF_n\right]
 = \bbP(W_\infty>0).
 \end{equation}
Finally taking the $n\to \infty$ limit in \eqref{trop} we obtain that almost surely
 $$ \ind_{\{W_{\infty}>0\}}\ge \ind_{\{\forall n, W_n>0\}} \bbP(W_\infty>0),$$
which implies {that either $ \bbP(W_\infty>0)$ or that $\ind_{\{W_{\infty}>0\}}= \ind_{\{\forall n, W_n>0\}}$ a.s.\ yielding the second claim in \eqref{tropic}}.
 \qed

 \medskip
 
 {\bf Acknowledgements:}
The author is grateful to Shuta Nakajima for comments on an early version of the manuscript, and to anonymous referees for helpful comments and the suggestion add Proposition \ref{auxil}.
He acknowledges the support of a productivity grand from CNQq and of a CNE grant from FAPERj.

\end{document}